\documentclass[11pt]{article}
\usepackage{amsmath,amsthm,amssymb,a4wide} 
\usepackage[margin=3cm]{geometry}
\usepackage{graphicx}
\usepackage{epsfig} 
\usepackage{slashed}  
\usepackage[hidelinks]{hyperref}
\usepackage[affil-it]{authblk}
\usepackage{lmodern}
\usepackage[T1]{fontenc}
\allowdisplaybreaks

\theoremstyle{plain}  
\newtheorem{theorem}{Theorem}[section]
\newtheorem{proposition}[theorem]{Proposition}
\newtheorem{lemma}[theorem]{Lemma}

\theoremstyle{remark}
\newtheorem{remark}[theorem]{Remark}

\numberwithin{equation}{section} 
\numberwithin{figure}{section}  
\usepackage{amssymb, amscd, mathrsfs, wasysym} 

\newcommand{\di}{\textrm{d}}

\newcommand \la \langle
\newcommand \ra \rangle

\newcommand \Kcal {\mathcal{K}}
\newcommand \Hcal {\mathcal{H}}

\newcommand \underdel {\underline \partial}

\newcommand \trianglerightNEW \triangleright

\newcommand \auth {\textsc}

\newcommand \bei {\begin{itemize}}
\newcommand \eei {\end{itemize}}
\newcommand \be {\begin{equation}}
\newcommand \bel {\be\label}
\newcommand \ee {\end{equation}}
\newcommand \bea {\be\aligned}
\newcommand \eea {\endaligned\ee}
\newcommand \del \partial
\newcommand \RR {\mathbb R}

\newcommand \eps \epsilon

\newcommand{\define}{:=}

\usepackage[most]{tcolorbox} 

\let\oldmarginpar\marginpar
\renewcommand\marginpar[1]{\-\oldmarginpar[\raggedleft\footnotesize #1]%
{\raggedright\footnotesize #1}}

\newcommand{\myfootnote}[1]{
    \renewcommand{\thefootnote}{}
    \footnotetext{\hspace{-16.5pt}\scriptsize#1}
    \renewcommand{\thefootnote}{\arabic{footnote}}
}

\begin{document}
\title{\bf \Large 
Stability of some two dimensional wave maps: \\ a wave--Klein-Gordon model} 
\author{Shijie Dong${}^\ast$
and Zoe Wyatt${}^\dagger$
}
%\affil{${}^\ast$Southern University of Science and Technology,\\ SUSTech International Center for Mathematics and \\ Department of Mathematics, 518055 Shenzhen,  China.}
%\affil{${}^\dagger$Department of Mathematics, King’s College London, \\ Strand Building, Strand, London, WC2R 2LS, UK.}
\date{}
\maketitle

\begin{abstract}
We are interested in the stability of a class of totally geodesic wave maps, as recently studied in \cite{Abbrescia-Chen, DM20, DMZ22}. The relevant equations of motion are a system of coupled semilinear wave and Klein-Gordon equations in $\RR^{1+n}$ whose nonlinearities have critical time decay when $n=2$.
In this paper we use a pure energy method to show  global existence when $n = 2$. By carefully examining the structure of the nonlinear terms, we are able to obtain uniform energy bounds at lower orders. This allows us to prove pointwise decay estimates and also to reduce the required regularity. 
\end{abstract}

%==============================================================================================
\section{Introduction}
\myfootnote{{\sl MSC codes:} 35L05, 35L52, 35L71 \\
{\sl Email addresses}: ${}^\ast$dongsj@sustech.edu.cn, shijiedong1991@hotmail.com, ${}^\dagger$zoe.wyatt@kcl.ac.uk, zw253@cam.ac.uk. \\
${}^\ast$Southern University of Science and Technology, SUSTech International Center for Mathematics and Department of Mathematics, 518055 Shenzhen,  China. \\
${}^\dagger$Department of Mathematics, King’s College London,  Strand Building, Strand, London, WC2R 2LS, UK.
}

%%%%%%%%%%%%%%%%%%%%%%%%%%%%%%%%%%%%
In this paper we study the global stability of certain solutions of the wave maps equation from Minkowski space $(\RR^{1+n}, \eta)$ to an $n-$dimensional space-form $(M, g)$ whose constant sectional curvature we denote by $\kappa_M$. Our problem comes from interesting recent work of Abbrescia and Chen \cite{Abbrescia-Chen}. We briefly describe their result here, and refer to the original paper \cite{Abbrescia-Chen} for full details. The authors study a class of wave maps which factorise as
\bel{eq:factorisation}
\RR^{1+n}\overset{\varphi_S}{\longrightarrow}\RR\overset{\varphi_I}\longrightarrow M.
\ee
Here $\varphi_S$ is a semi-Riemannian submersion to either $(\RR, \delta)$ or $(\RR, -\delta)$, and $\varphi_I$ is a Riemannian immersion from $(\RR, \delta)$ to $(M, g)$. The map $\varphi_S$ can be classified as spacelike or timelike depending on whether its codomain $\RR$ is equipped with the standard Euclidean metric $\delta$ or with $-\delta$.  The factorisation \eqref{eq:factorisation} implies that the wave map  $\varphi:= \varphi_I \circ \varphi_S$ is a totally geodesic wave map with infinite energy. The work \cite{Abbrescia-Chen} then considers the stability of $\varphi$ to small compact perturbations, which, after a careful expansion using geodesic normal
coordinates, amounts to proving stability for a coupled system of wave--Klein-Gordon equations in $\RR^{1+n}$. In particular, in \cite{Abbrescia-Chen} the following result is shown for $n \geq 3$:

\begin{theorem}\label{thm:WaveMap}
A totally geodesic map satisfying the factorisation \eqref{eq:factorisation} is globally asymptotically stable as a solution to the initial value problem for the wave maps equation under compactly supported smooth perturbations, provided that either (1) $\varphi_S$ is timelike and $\kappa_M <0$, or that (2) $\varphi_S$ is spacelike and $\kappa_M>0$. 
\end{theorem}

In the present paper we extend Theorem \ref{thm:WaveMap} to the critical case of $n=2$.

\subsection{Model Problem}
In the present paper we consider the following initial value problem for a system of coupled wave and Klein-Gordon variables $\{w, \vec{v}\}$ in $\RR^{1+2}$, where $\vec{v}=(v^1, \ldots, v^{p})$ for some $p \in \mathbb{N}$, obeying the equations:
\begin{subequations}\label{eq:model-wave-map}
\bel{eq:model-wave-map1}
\aligned
-\Box w &= P^\alpha \del_\alpha (|\vec{v}|^2) + \mathcal{F}^0(w, \vec{v}),
\\
-\Box v^i + v^i &= \widetilde{P}^\alpha v^i \del_\alpha w + \mathcal{F}^i (w, \vec{v}), \qquad i=1,\ldots, p
\endaligned
\ee 
where we define $\Box \define -\del_t^2 + \del_{x^1}^2 + \del_{x^2}^2$,  $|\vec{v}|^2 \define \sum_{i=1}^{p}(v^i)^2$ and also
\be
\aligned
\mathcal{F}^0(w, \vec{v}) &\define (A^{1})^{\alpha\beta}_i v^i \del_\alpha w \del_\beta w +  (A^{2})^\alpha_{ij} v^i v^j \del_\alpha w + (A^{3})^{\alpha\beta}_{ij} v^i \del_\alpha v^j \del_\beta w + (A^{4})_{ijk}  v^i v^j v^k 
\\
&\quad +(A^{5})^{\alpha}_{ijk}  v^i v^j \del_\alpha v^k + (A^6)^{\alpha\beta}_{ijk} v^i \del_\alpha v^j \del_\beta v^k,
\\
\mathcal{F}^i(w, \vec{v}) &\define  (B^{1})^{\alpha\beta i}_j v^j \del_\alpha w \del_\beta w + (B^2)^{\alpha i}_{jk}v^j v^k \del_\alpha  w + (B^{3})^{\alpha\beta i}_{jk} v^k \del_\alpha v^j \del_\beta w + (B^{4})^i_{jkl}v^j v^kv^l
\\
& \quad + (B^5)^{\alpha i}_{jkl} v^j v^k \del_\alpha v^l + (B^{6})^{\alpha\beta i}_{jkl} v^j \del_\alpha v^k \del_\beta v^l.
\endaligned
\ee
Furthermore $\{ P^\alpha, \widetilde{P}^\alpha, A^1, \ldots A^6, B^1, \ldots, B^6 \}$  are a collection of arbitrary constants which we do \emph{not} restrict to satisfy any null condition.
We have used Einstein summation convention above, where spacetime indices are represented by Greek letters $\alpha, \beta, \gamma\in\{0,1,2\}$ and Roman letters $i, j, k \in \{1, \ldots, p\}$ indicate a sum over components of $\vec{v}$. 

The initial data are prescribed on the slice $t=t_0=2$
\bel{eq:ID-wave-map}
\big( w, \del_t w, \vec{v}, \del_t \vec{v} \big)(2) = (w_0, w_1, \vec{v}_0, \vec{v}_1),
\ee
\end{subequations}
and these, as also assumed in \cite{Abbrescia-Chen},  are assumed to be compactly supported in $\{ (t, x) \in\RR^{1+2} : t=2, |x| \leq 1 \}$.

\subsection{Main statement and difficulties}

\begin{theorem}\label{thm1}
Consider in $\RR^{1+2}$ the initial value problem \eqref{eq:model-wave-map}
and let $N\geq 3$ be an integer. There exists $\eps_0>0$ such that for all $0<\eps<\eps_0$ and all compactly supported initial data satisfying the smallness condition
\bel{eq:thm-data}
\|w_0\|_{H^{N+1}(\mathbb{R}^2)} + \|w_1\|_{H^{N}(\mathbb{R}^2)}
+ \|\vec{v}_0\|_{H^{N+1}(\mathbb{R}^2)} + \|\vec{v}_1\|_{H^{N}(\mathbb{R}^2)}
\leq \eps,
\ee
the Cauchy problem \eqref{eq:model-wave-map} admits a global-in-time solution $(w, \vec{v})$ with pointwise decay estimates
\bel{eq:thm-decay}
|\del w(t,x)| \lesssim t^{-1/2}(t-|x|)^{-1/2}, 
\qquad
|\vec{v}(t,x)| \lesssim t^{-1},
\ee
as well as the almost sharp pointwise decay estimate
\bel{eq:thm-decay2}
|w(t,x)| \lesssim t^{-1/2+\delta/2}(t-|x|)^{-1/2+\delta/2}, 
\ee
with $0< \delta \ll 1$.
Furthermore, the lower-order energies are uniformly bounded. That is,
\bel{eq:thm-energy}
%E (\del^I L^J w, s)^{1/2} 
%+ 
%E_1 (\del^I L^J \vec{v}, s)^{1/2} 
\big\| (s/t) \del Z^I w \big\|_{L^2_f(\Hcal_s)} + \big\| (s/t) \del Z^I \vec{v} \big\|_{L^2_f(\Hcal_s)} + \big\| Z^I \vec{v} \big\|_{L^2_f(\Hcal_s)}
\lesssim 1,
\ee
for $|I| \leq N-1$ (see also the definitions in Section \ref{subsec:energies}).
\end{theorem}

\begin{remark}
The smallest integer for the order of regularity with which one can conduct $L^2$--$L^\infty$ energy estimates in dimension $n=2$ is $N=3$. We reach this low regularity in Theorem \ref{thm1} by proving uniform energy bounds and carefully analysing the types of nonlinear terms appearing in the PDEs. The strategy we used to reduce regularity requirement can also be applied to the study of other nonlinear equations, for instance \cite{Dong2006}.
Also, we note that the paper \cite{Abbrescia-Chen} proves an analogous version of Theorem \ref{thm1} in dimension $n=3$ with $N \geq 3$. Given the slower decay of wave and Klein-Gordon fields in two spatial dimensions, it is particularly interesting that we can also obtain global existence down to $N = 3$ when $n=2$. 
We also refer to \cite{DM20, DMZ22} for the study of two dimensional totally geodesic wave maps.
\end{remark}

\begin{remark}
To simplify notation in the present paper, when we write $\vec{v}$ in an estimate it should be understood that the estimate holds for each component $v^i$ for $i = 1, \ldots, p$. 
\end{remark}

\begin{remark}
One key feature of the PDE system \eqref{eq:model-wave-map} is the divergence structure appearing in the quadratic part of the wave nonlinearity. Another key feature is the fact that all other quadratic and cubic terms appearing in the wave and Klein-Gordon nonlinearities involve at least one undifferentiated Klein-Gordon factor. See also Remark \ref{rem:nonlinear-structure}. 
\end{remark}

\paragraph{Critical nonlinearities and difficulties.}
The slow decay of linear wave and linear Klein-Gordon fields in two spatial dimensions makes the proof of global existence and decay for the PDE system \eqref{eq:model-wave-map} difficult. In particular the quadratic nonlinearities are critical when $n=2$, in the sense that by assuming $w$ and $v^i$ obey linear decay estimates, the best we can expect for the nonlinearities (in the flat $t=$const slicing) is
$$
\| P^\alpha \del_\alpha (|\vec{v}|^2) \|_{L^2(\RR^2)}\lesssim t^{-1},
\qquad 
\| \widetilde{P}^\alpha \vec{v} \del_\alpha w \|_{L^2(\RR^2)}\lesssim t^{-1}.
$$
Similarly, in the hyperboloidal foliation with hyperboloidal time $s = \sqrt{t^2 - |x|^2}$, the best we can expect (using the Klainerman-Sobolev inequality) is
\bel{Intro-Naive-Est}
 \aligned
\| (s/t) \del w \|_{L^2_f(\Hcal_s)} &\lesssim 1,
\qquad
&|\del w| &\lesssim t^{-1/2} (t-r)^{-1/2} \lesssim s^{-1},
\\
\| v \|_{L^2_f(\Hcal_s)} &\lesssim 1,
\qquad
&|v| &\lesssim t^{-1}.
\endaligned
\ee
Here $\Hcal_s$ are constant $s$-surfaces defined in \eqref{def:prelioms2} and $L^2_f(\Hcal_s)$ is defined in \eqref{flat-int}.
Naive calculations then lead us to the estimates
$$
\aligned
\| P^\alpha \del_\alpha (|\vec{v}|^2)\|_{L^2_f(\Hcal_s)} 
&\lesssim \| (s/t) \del \vec{v} \|_{L^2_f(\Hcal_s)} \| (t/s) \vec{v} \|_{L^\infty(\Hcal_s)}
\lesssim s^{-1},
\\
\| \widetilde{P}^\alpha \vec{v} \del_\alpha w\|_{L^2_f(\Hcal_s)} 
&\lesssim \|(s/t) \del w\|_{L^2_f(\Hcal_s)} \| (t/s) \vec{v}\|_{L^\infty(\Hcal_s)} 
\lesssim s^{-1}.
\endaligned
$$
We see both quantities are at the borderline of integrability. For pure wave equations, such borderline nonlinearities can lead to finite-time blow-up, see for example \cite{John2}.

Global existence for the model problem \eqref{eq:model-wave-map} with small compactly supported initial data with regularity $N \geq 7$ was recently established by Duan and Ma in \cite{DM20}. In this work, the authors used the $L^2-L^\infty$ hyperboloidal energy method and allowed their energy to slowly grow at all orders. Due to this growth, they obtained decay for the differentiated wave component $|\del w|$  through direct  $L^\infty-L^\infty$ estimates of the fundamental solution to the wave equation. 

In the present paper we use a simpler analysis based on a pure energy method, which we believe provides an interesting alternative to the proof of \cite{DM20}. In particular, we show that the energy is uniformly bounded at all orders except the top-order, see \eqref{eq:thm-energy}. This is somewhat surprising given the criticality of the nonlinearities, and relies on a careful analysis of the type of terms appearing in the nonlinearity, see Remark \ref{rem:nonlinear-structure}. The uniform control then allows us to lower the required regularity, and obtain decay estimates \eqref{eq:thm-decay}  for $|\del w|$ from the standard Klainerman-Sobolev inequalities. 

One of the key ideas in our proof is to use an $L^1$ version of the classical energy inequality, given in \eqref{eq:w-EE-L1}. This allows us to redistribute various $t,r$ weights and thus exploit the additional $(t-r)$-decay enjoyed in \eqref{Intro-Naive-Est} by the wave component. Using this for the wave equation, for example at lowest order, we find:
\bel{eq:intro10}
\aligned
\|P^\alpha \del_\alpha(|\vec{v}|^2)\cdot (s/t)\del_t w \|_{L^1_f(\Hcal_{s})}
&\lesssim
\| \del \vec{v} \|_{L^2_f(\Hcal_{s})}
\| (t-r)(s/t) \del w\|_{L^2_f(\Hcal_{s})}
\| (t-r)^{-1} \vec{v} \|_{L^\infty_f(\Hcal_{s})}
\\
&
\lesssim s^{-2}.
\endaligned
\ee
The energy inequality at lower-orders for the Klein-Gordon components is similar, since we see that the term that needs to be controlled, namely
$$
\|\widetilde{P}^\alpha \vec{v} \del_\alpha w\cdot (s/t)\del_t \vec{v} \|_{L^1_f(\Hcal_{s})},
$$
has exactly the same structure as in \eqref{eq:intro10}. 

Another new feature of our proof is the estimate \eqref{eq:thm-decay2} on the undifferentiated wave component. We are able to prove this estimate by using the divergence type structure in the quadratic nonlinearity appearing in the wave equation in \eqref{eq:model-wave-map1}. This allows us to perform a transformation, due to Katayama \cite{Katayama12a}, of the form
$$
w=W_1+P^\alpha \del_\alpha W_2.
$$
Both $W_1$ and $W_2$ obey  inhomogeneous wave equations with  fast decaying nonlinearities, see \eqref{eq:aux-wave-map}.
We then obtain $L^2$ control on $W_1$, and subsequent $L^\infty$ control, via a conformal energy estimate which is used to prove \eqref{eq:thm-decay2}.

\paragraph{Theorem \ref{thm:WaveMap} and relationship to the Model Problem.}
The equations of motion relevant to Theorem \ref{thm:WaveMap} have been derived by Abbrescia and Chen in \cite[\S 3.1.1, \S 3.1.2]{Abbrescia-Chen}. 
 For example, in the case (1) of Theorem \ref{thm:WaveMap}, the relevant perturbation variables $(u^1, u^2, \ldots, u^n)$ of the wave map obey the equations:
\begin{align}\notag
\Box u^1 &= -2 \sum_{i=2}^n u^i \del_t u^i - 2\sum_{i, j=2}^n \del^2_{ij} \Gamma^1_{k1}(\ell, \vec{0}) \cdot u^i u^j \cdot \eta(\di u^k, \di \ell)
\\
&\quad + \sum_{i=2}^n \del_i \Gamma^1_{jk}(\ell, \vec{0}) \cdot u^i \cdot\eta(\di u^j, \di u^k) - \sum_{i,j,k=2}^n \del^3_{ijk} \Gamma^1_{11}(\ell, \vec{0}) \cdot u^i u^j u^k
 + \textrm{l.o.t},\notag
\\
\Box u^p - u^p &= 2 u^p \del_t u^1 - 2\sum_{i,j=2}^n \del^2_{ij} \Gamma^p_{k1}(\ell, \vec{0}) \cdot u^i u^j \cdot \eta(\di u^k, \di\ell)\label{Intro:WaveMapEq} 
\\
&\quad + \sum_{i=2}^n \del_i \Gamma^1_{jk}(\ell, \vec{0}) \cdot u^i \cdot \eta(\di u^j, \di u^k) - \sum_{i,j,k=2}^n \del^3_{ijk} \Gamma^p_{11}(\ell, \vec{0}) \cdot u^i u^j u^k
+ \textrm{l.o.t},\notag
\end{align}
for $p = 2, \ldots, n$ with $\ell\equiv t$ and l.o.t denoting lower order terms. For definitions of the terms appearing here, see \cite{Abbrescia-Chen}. 
Case (2) of Theorem \ref{thm:WaveMap} gives a similar system of equations. 
A key point is that the Christoffel symbols and derivatives thereof appearing in \eqref{Intro:WaveMapEq}  can be treated as universal constants \cite[Lemma 2.5]{Abbrescia-Chen}. Thus, one sees structurally the similarity between equations \eqref{eq:model-wave-map} and  the model equations \eqref{eq:model-wave-map}. 
Full details of the relationship between the PDE system relevant to Theorem \ref{thm:WaveMap}   and the model problem treated in Theorem \ref{thm1} can be found in \cite{Abbrescia-Chen} and also \cite[\textsection5]{DM20}. 
We finally remark that the relationship  between these two PDE systems does not degenerate with $n$ in any way. Since linear dispersion for  wave and Klein-Gordon fields improves as $n$ increases, our arguments in Theorem \ref{thm1} also hold trivially for $n\geq 3$, and thus establish Theorem \ref{thm:WaveMap} for all $n\geq 2$.

\paragraph{The hyperboloidal foliation in $\RR^{1+2}$.}
In the present paper we study the system \eqref{eq:model-wave-map} using the vector-field method adapted to a hyperboloidal foliation of Minkowski spacetime. This method originates in work of Klainerman \cite{Klainerman85}, see also  H\"ormander \cite{Hormander}, in the context of Klein-Gordon equations. We also note the pioneering work on Strichartz estimates in the hyperbolic space by Tataru \cite{Tataru} in the context of wave equations. The method was later reintroduced to establish global-in-time existence results for nonlinear systems of coupled wave and
Klein-Gordon equations by LeFloch and Ma in  \cite{PLF-YM-book}  under the name of the ``hyperboloidal foliation method''. 
Recently, Ma \cite{Ma2017b} has initiated the study of coupled nonlinear wave and Klein-Gordon systems in $\RR^{1+2}$ using the hyperboloidal foliation method. This has led to  global existence results under the assumption of compact initial data for a large class of wave--Klein-Gordon(-type) systems including null forms and other interesting nonlinearities, see for example \cite{Dong1912, Dong2006, DW-2020, DW-21, DM20, Ma2019} and references cited within.

\paragraph{Outline.} The rest of this paper is organised as follows. In Section \ref{sec:prelim} we revisit some basic notations and other preliminaries of hyperbolic equations and the vector field method. Then Theorem \ref{thm1} is proved by using a standard bootstrap argument in Section \ref{sec:Proof}.

\section{Preliminaries}\label{sec:prelim}

We first introduce some basic notations in the framework of the vector field method. We adopt the signature $(-, +, +)$ in the $(1+2)$--dimensional Minkowski spacetime $(\mathbb{R}^{1+2}, \eta=\text{diag}(-1,+1,+1)$, and for the point $(t, x) = (x^0, x^1, x^2)$ in Cartesion coordinates we denote its spatial radius by $r \define | x | = \sqrt{(x^1)^2 + (x^2)^2}$. Spacetime indices are represented by Greek letters $\alpha, \beta, \gamma\in\{0,1,2\}$ while spatial indices are denoted by Roman letters $a, b, c\in\{1,2\}$. In our analysis we will frequently use the vector fields
$$
\del_\alpha \define \del_{x^\alpha}, \quad L_a\define  x_a \del_t + t \del_a,
\quad L_0 \define t\del_t + x^a \del_a,
$$
where $x_a:= \delta_{ab}x^b$. These are respectively called the translation vector fields, the Lorentz boosts and the scaling vector field. 
We also introduce the rotation vector fields $
\Omega_{ab}
\define x_a \del_b - x_b \del_a$.

 We next state some notation concerning the hyperboloidal foliation method used in \cite{PLF-YM-book}. Let $s_0 \geq 1$.
 Throughout the paper, we consider functions defined in the interior of the future light cone $\Kcal$, with vertex $(1, 0, 0)$ and boundary $\del\Kcal$:
\bea \label{def:prelioms}
\Kcal&\define \{(t, x): r< t-1 \},
\\
\del\Kcal&\define\{(t, x): r= t-1 \}.
\eea
We will consider hyperboloidal hypersurfaces $\Hcal_s$ with $s>1$ foliating the interior of $\Kcal$. We define $\Kcal_{[s_0, s_1]}$ to denote subsets of $\Kcal$ limited by two hyperboloids $\Hcal_{s_0}$ and $\Hcal_{s_1}$ with $s_0 \leq s_1$, and let $\del\Kcal_{[s_0, s_1]}$ denote the conical boundary. In summary:
\bea \label{def:prelioms2}
\Hcal_s&\define \{(t, x): t^2 - r^2 = s^2 \}, \quad s>1,
\\
\Kcal_{[s_0, s_1]} &\define \{(t, x)\in\Kcal: s_0^2 \leq t^2- r^2 \leq s_1^2\}, \quad s_0 \leq s_1,
\\
\del\Kcal_{[s_0, s_1]}&\define \{(t, x): s_0^2 \leq t^2- r^2 \leq s_1^2; r=t-1 \}.
\eea
For a point $(t, x) \in \Hcal_s \cap \Kcal$, we note the following relations
\bel{eq:s-t-identities}
|x| \leq t,
\qquad
s \leq t \leq {s^2 +1\over 2}.
\ee

The semi-hyperboloidal frame is defined by
$$
\underdel_0\define \del_t, \qquad \underdel_a\define {L_a \over t} = {x_a\over t}\del_t+ \del_a.
$$
Note that the vectors $\underdel_a$ generate the tangent space to the hyperboloids. The partial derivatives can also be expressed by the semi-hyperboloidal frame, which read
$$
\del_t = \underdel_0,
\qquad
\del_a = - {x_a\over t} \underdel_0 + \underdel_a.
$$
%We also introduce the vector field $\underdel_\perp\define \del_t+\frac{x^a}{t}\del_a = \frac{L_0}{t} $
%which is orthogonal to the hyperboloids.
%For the semi-hyperboloidal frame $\{\underdel_0, \underdel_a\}$ above, the dual frame is given by $\underline{\theta}^0\define dt- (x^a / t)dx^a$ and $\underline{\theta}^a\define dx^a$. The (dual) semi-hyperboloidal frame and the (dual) natural Cartesian frame are connected by the relations
%\bel{semi-hyper-Cts}
%\underdel_\alpha= \Phi_\alpha^{\alpha'}\del_{\alpha'}, \quad \del_\alpha= \Psi_\alpha^{\alpha'}\underdel_{\alpha'}, \quad \underline{\theta}^\alpha= \Psi^\alpha_{\alpha'}dx^{\alpha'}, \quad dx^\alpha= \Phi^\alpha_{\alpha'}\underline{\theta}^{\alpha'},
%\ee
%where the transition matrix ($\Phi^\beta_\alpha$) and its inverse ($\Psi^\beta_\alpha$) are given by
%\be
%(\Phi_\alpha^{ \beta})=
%\begin{pmatrix}
%1 & 0 &   0   \\
%{x^1 / t} & 1  & 0    \\
%{x^2 / t} &  0  &  1   
%\end{pmatrix}
%, \quad
%(\Psi_\alpha^{ \beta})=
%\begin{pmatrix}
%1 & 0 &   0   \\
%-{x^1 / t} & 1  & 0   \\
%-{x^2 / t} &  0  &  1 
%\end{pmatrix}.
%\ee
%Finally we note the identities
%\bel{deriv-identities}
%\del_t
%=
%{t^2 \over s^2} \left( \underdel_\perp -\frac{x^a}{t} \underdel_a \right),
%\qquad
%\del_a
%=
%-{t x^a \over s^2} \underdel_\perp + {x^a x^b \over t^2} \underdel_b + \underdel_a,
%\ee
%which can be used to deduce pointwise decay of $\del_\alpha w$ for a wave component $w$. 
We write the following decomposition, valid on $\Hcal_s$, of the flat wave operator into the semi-hyperboloidal frame:  
\bea \label{eq:wavedecomp}
-\Box = \del_t^2 - \sum_{a=1}^2 \del_a^2 
&= \left(\frac{s}{t}\right)^2 \underline{\del}_0 \underline{\del}_0+2\frac{x^a}{t} \underline{\del}_a\underline{\del}_0
-\sum_{a=1}^2 \underline{\del}_a\underline{\del}_a 
- \frac{r^2}{t^3}\underline{\del}_0+\frac{2}{t}\underline{\del}_0
\\
&= \left(\frac{s}{t}\right)^2 \del_t \del_t 
+t^{-1} \left( 2\frac{x^a}{t} L_a \del_t
-\sum_{a=1}^2 L_a \underline{\del}_a 
- \frac{r^2}{t^2}\del_t+2\del_t \right).
\eea

\paragraph*{Standard Notation.}
Throughout the paper, we use $A\lesssim B$ to denote that there exists a generic constant $C>0$ such that $A\leq BC$. We adopt the Einstein summation convention although sometimes explicitly write out summands. For the ordered set $\{ Z_i\}_{i=1}^5\define\{ \del_{0},\del_1, \del_2, L_1,L_2\}$, and for any multi-index $I=(\alpha_1, \ldots, \alpha_5)$ of length $|I|\define \alpha_1 + \ldots + \alpha_5 = m$ we denote by $Z^I$ the $m$-th order vector field $Z^I\define\Gamma_1^{\alpha_1}\ldots\Gamma_5^{\alpha_5}$. A similar definition holds for $L^I$ where we only allow for $L_1, L_2$ to appear For $x \in \RR$ we write $\lfloor x \rfloor$ to denote the greatest integer less than or equal to $x$.

%==============================================================================================

\subsection*{Auxiliary tools}\label{sec:Tools}
\subsection{Standard energy estimates}\label{subsec:energies}

Following \cite{Hormander, PLF-YM-book}, we introduce the energy functional $E_m$, in a Minkowski background, for a function $\phi$ defined on a hyperboloid $\Hcal_s$: 
\bel{eq:2energy} 
\aligned
E_m(s, \phi)
&\define
 \int_{\Hcal_s} \Big( \big(\del_t \phi \big)^2+ \sum_a \big(\del_a \phi \big)^2+ 2 (x^a/t) \del_t \phi \del_a \phi + m^2 \phi ^2 \Big) \, \di x
\\
               &= \int_{\Hcal_s} \Big( \big( (s/t)\del_t \phi \big)^2+ \sum_a \big(\underdel_a \phi \big)^2+ m^2 \phi^2 \Big) \, \di x
                \\
               &= \int_{\Hcal_s} \Big( \big( t^{-1} L_0 \phi \big)^2+ \sum_a \big( (s/t)\del_a \phi \big)^2+ \big( t^{-1}\Omega_{12} \phi \big)^2+ m^2 \phi^2 \Big) \, \di x.
 \endaligned
 \ee
The last two expressions of the energy functional $E_m(s, \phi)$ in \eqref{eq:2energy} imply
\be 
\int_{\Hcal_s} \Big( \big( (s/t)\del_t \phi \big)^2 + \sum_a \big( (s/t)\del_a \phi \big)^2 \Big) \, \di x
\leq 2 E_m(s, \phi).
\ee
 
In the massless case we denote $E(s, \phi)\define E_0(s, \phi)$ for simplicity.
In the above, the integral in $L^p_f(\Hcal_s)$ is defined from the standard (flat) metric in $\RR^2$, i.e. 
\bel{flat-int}
\|\phi \|_{L^p_f(\Hcal_s)}^p
\define\int_{\Hcal_s}|\phi |^p \, \di x 
=\int_{\RR^2} \big|\phi(\sqrt{s^2+r^2}, x) \big|^p \, \di x,
\qquad
p \in [1, +\infty).
\ee

\begin{proposition}[Energy estimates]\label{prop:BasicEnergyEstimate}
Let $m\geq 0$ and $\phi$ be a sufficiently regular function defined in the region $\Kcal_{[s_0, s]}$, vanishing near $\del\Kcal_{[s_0, s]}$ and satisfying
$$
-\Box \phi + m^2 \phi = f.
$$
For all $s \geq s_0$, it holds that
\bel{eq:w-EE} 
E_m(s, \phi)^{1/2}
\leq 
E_m(s_0, \phi)^{1/2}
+  \int_{s_0}^s \| f\|_{L^2_f(\Hcal_{\tau})} \, \di \tau.
\ee
Similarly for all $s \geq s_0$, it holds that
\bel{eq:w-EE-L1} 
E_m(s, \phi)
\leq 
E_m(s_0, \phi)
+ 2 \int_{s_0}^s \|f\cdot (\tau/t)\del_t \phi  \|_{L^1_f(\Hcal_{\tau})} \, \di \tau.
\ee
\end{proposition}

\begin{remark}
The proof of both these energy estimates can be found in \cite[Proposition 2.3.1]{PLF-YM-book}.
The first energy estimate above is a classical, frequently used result. The second estimate, where the sourcing term appears in $L^1$ instead of $L^2$, is of course widely known yet perhaps not as frequently used. We emphasise that it plays a crucial role in the present paper. Another situation where the inhomogeneity is estimated in $L^1$ is the ghost-weight energy inequality of Alinhac \cite{Alinhac01b}. This has recently been applied to the Klein-Gordon--Zakharov equations in two spatial dimensions in \cite{Dong2006}.
\end{remark}

\subsection{Conformal energy estimates}
We now introduce a conformal-type energy which was adapted to the hyperboloidal foliation setting by Ma and Huang in three spatial dimensions in \cite{YM-HH}, see also Wong for work in two spatial dimensions \cite{Wong}. A key part of the following proposition, due to Ma \cite{Ma2019}, is in giving an estimate for the weighted $L^2$ norm $\| (s/t) \phi \|_{L^2_f(\Hcal_s)}$ for a wave component $\phi$.

\begin{proposition}[Conformal energy estimates]\label{lem:ConformalEnergy}
Let $\phi$ be a sufficiently regular function defined in the region $\Kcal_{[s_0, s]}$ and vanishing near $\del\Kcal_{[s_0, s]}$. Define the conformal energy
$$ 
E_{con} (s,\phi)
\define
\int_{\Hcal_s} \Big( \sum_a \big( s \underdel_a \phi \big)^2 + \big( K \phi + \phi \big)^2 \Big) \, \di x,
$$
in which we used the notation of the weighted inverted time translation
$$
K \phi 
\define \big( s \del_s + 2 x^a \underdel_a \big) \phi,
\qquad
\del_s \define {s\over t} \del_t.
$$
Then for all $s \geq s_0$ we have the energy estimate
\bel{eq:con-estimate} 
E_{con} (s,\phi)^{1/2}
\leq 
E_{con} (s_0, \phi)^{1/2}
+
2 \int_{s_0}^s \tau \| \Box \phi \|_{L^2_f(\Hcal_{\tau})} \, \di\tau.
\ee
Furthermore for all $s \geq s_0$ we have 
\bel{eq:l2type-wave} 
\|(s/t) \phi \|_{L^2_f(\Hcal_s)} \lesssim \|(s_0/t) \phi \|_{L^2_f(\Hcal_{s_0})} + \int_{s_0}^s \tau^{-1} E_{con}(\tau, \phi)^{1/2} \di \tau.
\ee
\end{proposition}

\begin{proof}
The proof of \eqref{eq:con-estimate} and \eqref{eq:l2type-wave} follows from the differential identities 
$$
\aligned
& s \big( s\del_s u + 2x^a \underline{\del}_a u + u \big) \big( -\Box u \big)
\\
=
& {1\over 2} \del_s \big( s\del_s u + 2x^a \underline{\del}_a u + u \big)^2
+
{1\over 2} \del_s \big( s^2 \underline{\del}_b u \underline{\del}{}^b u  \big) 
-
s^2 \underline{\del}_b \big( \del_s u \underline{\del}{}^b u \big) 
-
2 s \overline{\del}_b \big( x^a  \underline{\del}_a u \underline{\del}{}^b u \big)
\\
+& s  \underline{\del}_a \big( x^a \underline{\del}_b u \underline{\del}{}^b u \big)
- s \underline{\del}_b \big( u \underline{\del}{}^b u \big).
\endaligned
$$
and
$$
\aligned
s \del_s \big( (s/t)^2 u^2 \big) + \underdel_a \big( x^a (s/t)^2 u^2 \big) 
=
2 (s / t)^2 u \big( Ku + u \big) - 2 (s / t)^2 u x^a \underdel_a u,
\endaligned
$$
respectively. The remaining details can be found in \cite{Ma2019}.
\end{proof}

Note that in three spatial dimensions one can use a Hardy estimate adapted to hyperboloids which is stronger than the estimate \eqref{eq:l2type-wave}. 

In conjunction with the above, we are also able to bound the scaling vector field in the proposition below. This is seen by writing the scaling
vector field in the semi-hyperboloidal frame as
$$
\frac{s}{t} L_0 \phi = \frac{s}{t}(K\phi + \phi) - \frac{x^a}{t}s \underdel_a\phi - \frac{s}{t}\phi.
$$

\begin{proposition}\label{prop:L2}
Let $\phi$ be a sufficiently regular function defined in the region $\Kcal_{[s_0, s]}$ and vanishing near $\del\Kcal_{[s_0, s]}$, the following estimate on the scaling vector field holds
$$ 
\big\| (s/t) L_0 \phi \big\|_{L^2_f(\Hcal_s)}
\lesssim
\big\| (s/t) \phi \big\|_{L^2_f(\Hcal_s)}
+
E_{con}(s, \phi)^{1/2}.
$$
\end{proposition}

Based on the bound on the scaling vector field, we can then make the following estimate, which crucially allows us to gain more $(t-r)$--decay: 
\bel{eq:202} 
\big| \del \phi \big|
\lesssim
(t-r)^{-1} \big( \big| L_0 \phi \big| + \sum_a \big| L_a \phi \big| \big).
\ee
Note that \eqref{eq:202} derives from an identity in \cite[\textsection II, Proposition 1.1]{Sogge}  and also relies  on our ability to estimate the angular momentum operators $x_a\del_b-x_b\del_a$ in terms of the Lorentz boosts $L_a$ in the support of the solution.

\subsection{Commutator estimates}
We first have the following identities
$$
\aligned
\left[ \del_t, L_a \right]&=\del_a , \quad 
[\del_b, L_a] = \delta_{ab} \del_t, \quad
[t, L_a] = -x_a , \quad
[x_b, L_a] = -t \delta_{ab},
\\
[ L_a, L_b ] &= x_a \del_b -x_b \del_a =t^{-1}(x_aL_b - x_bL_a).
\endaligned
$$
By using these identities and writing $L_0 = t\del_t+x^b\del_b$ and $K = t\del_t + x^b \del_b + (x^b/t)L_b$ we find
\bea\label{eq:CommutatorsK}
[\del_\alpha, L_0]&=\del_\alpha, \quad &[L_a, L_0]&=0,
\quad
[L_a, K] = (s/t)^2 L_a, \quad &[\del_a, K]&=(2/t)L_a.
\eea
%Finally we have the useful property that for the $Q_0$ null form 
%\bea \label{eq:CommutatorsQzero}
%L_a Q_0( f,  g) &= Q_0((L_a f),  g) + Q_0( f, L_a g), \\
%\del_\alpha Q_0( f,  g) &= Q_0(\del_\alpha f,  g) + Q_0( f, \del_\alpha g) .
%\eea

The following lemma allows us to control the commutators. It is proven in \cite[\textsection3]{PLF-YM-book}.
\begin{lemma} \label{lem:est-comm}
Let $\phi$ be a sufficiently regular function supported in the region $\mathcal{K}$. Then, for any multi-indices $I$, there exist generic constants $C=C(|I|)$ such that
\begin{subequations}\label{eq:est-cmt}
\begin{align}
\label{eq:est-cmt1}
 \big| [Z^I, \del_\alpha] \phi \big| 
&\leq 
C \sum_{|I'|<|I|} \sum_\beta \big|\del_\beta Z^{I'} \phi \big|,
\\
%\label{eq:est-cmt2}
%& \big| [\del^I L^J, \underdel_a] \phi \big| 
%\leq 
%C \Big( \sum_{| I' |<| I |, | J' |< | J |} \sum_b \big|\underdel_b \del^{I'} L^{J'} \phi \big| + t^{-1} \sum_{| I' |\leq | I |, |J'|\leq |J|} \big| \del^{I'} L^{J'} \phi \big| \Big),
%\\
%\label{eq:est-cmt3}
%& \big| [\del^I L^J, \underdel_\alpha] \phi \big| 
%\leq 
%C\Big( \sum_{| I' |<| I |, | J' |< | J |}\sum_\beta \big|\del_\beta \del^{I'} L^{J'} \phi \big| + t^{-1} \sum_{| I' |\leq | I |, |J'|\leq |J|} \sum_\beta \big| \del_\beta \del^{I'} L^{J'} \phi \big| \Big),
%\\
%\label{eq:est-cmt4}
%& \big| [\del^I L^J, \del_\alpha \del_\beta] \phi \big| 
%\leq 
%C \sum_{| I' |\leq  | I |, |J'|<|J|} \sum_{\gamma, \gamma'} \big| \del_\gamma \del_{\gamma '} \del^{I'} L^{J'} \phi \big|,
%\\
\label{eq:est-cmt6}
 \big| Z^I ((s/t) \del_\alpha \phi) \big| 
&\leq 
|(s/t) \del_\alpha Z^I \phi| + C \sum_{| I'|<| I |}\sum_\beta \big|(s/t) \del_\beta Z^{I'} \phi \big|.
\end{align}
\end{subequations}
%\bel{eq:est-cmt5}
%\aligned
%| [\del^I L^J, \underdel_a \underdel_\beta] u | + | [\del^I L^J, \underdel_\alpha \underdel_b] u | 
%&\leq 
%C(| I |, |J|) \Big( \sum_{| I' |\leq  | I |, |J'|<|J|, c, \gamma } | \under\del_c \underdel_{\gamma} \del^{I'} L^{J'} u |
%\\ 
%& + t^{-1} \sum_{| I' |<  | I |, |J'|\leq |J|, c, \gamma } | \underdel_c \underdel_{\gamma } \del^{I'} L^{J'} u | 
%\\
%& + t^{-1} \sum_{| I' |\leq  | I |, |J'|\leq |J|, \gamma } | \del_{\gamma } \del^{I'} L^{J'} u | \Big),
%\endaligned
%\ee
Recall here that Greek indices $\alpha, \beta \in \{0,1,2\}$ and Roman indices $a,b \in \{1,2\}$. 
\end{lemma}

%-------------------------------------------------------------------------------------

%
%Finally we end with the following short technical lemma, whose proof can be found in \cite{Ma2018}.
%\begin{lemma}
%In the cone $\Kcal$, there exists a constant $C>0$, determined by $I$ and $J$, such that
%\bel{eq:Est-s/t}
%|\del^I L^J(s/t)| \leq 
%     \begin{cases}
%      C(s/t) &\quad |I|=0,\\
%      Cs^{-1} &\quad |I|>0.
%     \end{cases}
%\ee
%\end{lemma}

\subsection{Pointwise Estimates}
We now state a Klainerman-Sobolev estimate in terms of the hyperboloidal coordinates. The proof is standard and can be found in \cite[\textsection 7]{Hormander} or \cite[\textsection 5]{PLF-YM-book}.

\begin{lemma}[Sobolev Estimate] \label{lem:sobolev}
For all sufficiently smooth functions $\phi= \phi(t, x)$ supported in $\Kcal$ and for all  $s \geq 2$, there exists a constant $C>0$ such that
\bel{eq:Sobolev2}
\sup_{\Hcal_s} \big| t \phi(t, x) \big|  \leq C \sum_{| J |\leq 2} \| L^J \phi \|_{L^2_f(\Hcal_s)}.
\ee
Furthermore we have
\bel{eq:Sobolev3}
\sup_{\Hcal_s} \big| s \phi(t, x) \big|  \leq C \sum_{| J |\leq 2} \| (s/t) L^J \phi \|_{L^2_f(\Hcal_s)},
\ee
%\be\label{eq:weightedSobolev}
%\sup_{\Hcal_s} \big| s (t-r)^{-\gamma} \phi(t, x) \big|  \lesssim \sum_{| J |\leq 2} \| (s/t) (t-r)^{-\gamma} L^J \phi \|_{L^2_f(\Hcal_s)}.
%\ee
\end{lemma}

Using the decomposition \eqref{eq:wavedecomp}, it is fairly straightforward to derive  the following lemma. We refer to \cite[\textsection8.1 and \textsection8.2]{PLF-YM-book} for a proof. 

\begin{lemma}\label{lem:Wavedeldel}
Suppose $\phi$ is a $C^2$ function supported in $\Kcal_{[s_0, s]}$  vanishing near $\del\Kcal$ satisfying
$$
\Box \phi = f.
$$
Then there exists a constant $C>0$ such that for all $\alpha, \beta\in\{0,1,2\}$
$$
|\del_\alpha \del_\beta \phi| \leq C\left(  \frac{1}{t-|x|} \left( |\del L \phi| + |\del \phi| \right) + \frac{t}{t-|x|} |f| \right).
$$
\end{lemma}

\section{The bootstrap argument}\label{sec:Proof}
As shown in \cite[\textsection11]{PLF-YM-book}, initial data posed on the hypersurface $\{t_0=2\}$ and localised in the unit ball $\{x\in\RR^2:|x|\leq 1\}$ can be developed as a solution of the PDE to the initial hyperboloid $\Hcal_{s_0}$ with the smallness conserved, where we have put $s_0=2$.
Thus by the definition of the hyperboloidal energy functional and smallness of the data from \eqref{eq:thm-data}, there exists a constant $C_0>0$ such that on the hyperboloid $\Hcal_{s_0}$ the following energy bounds hold for all $|I|\leq N$:
\bea\label{eq:BApre}
E(s_0, Z^I w)^{1/2}  + E_1 (s_0, Z^I \vec{v})^{1/2} \leq C_0 \eps. 
\eea
Fix $0<\delta\ll1$ a constant. 
We make the following bootstrap assumptions on the interval $[s_0, s_1)$
\begin{subequations} \label{eq:BA-Easy}
\begin{align}
E(s, Z^I w)^{1/2} + E_1 (s, Z^I \vec{v})^{1/2}
&\leq C_1 \eps s^\delta,
\quad
&|I| &\leq N, \label{eq:BA-Easy1}
\\
E(s, Z^I w)^{1/2} + E_1 (s, Z^I \vec{v})^{1/2}
&\leq C_1 \eps,
\quad
&|I|&\leq N-1, \label{eq:BA-Easy2}
\\
|\del Z^I w|
&\leq C_1 \eps s^{-1+\delta} (t-r)^{-1},
\quad
&|I| &\leq N-3, \label{eq:BA-Easy3}
\end{align}
\end{subequations}
in which $C_1 \gg C_0$, and $\eps \ll 1$ such that $C_1 \eps \ll 1$, and $s_1$ is defined as
$$
s_1 := \sup \{ s : s > s_0, \, \eqref{eq:BA-Easy}\,\, holds\}.
$$

By the definition of the energy functional $E_1(s, \cdot), E(s, \cdot)$, the bootstrap assumptions immediately imply the following $L^2$ estimates.

\begin{lemma}\label{lem:L2-BA}
Under the bootstrap assumptions \eqref{eq:BA-Easy}, we have
$$
\aligned
\big\| (s/t) \del Z^I w \big\|_{L^2_f(\Hcal_s)} + \big\| (s/t) \del Z^I \vec{v} \big\|_{L^2_f(\Hcal_s)} + \big\| Z^I \vec{v} \big\|_{L^2_f(\Hcal_s)}
&\lesssim C_1 \eps s^\delta,
\hspace{3pt}
&|I| \leq N,
\\
\big\| (s/t) \del Z^I w \big\|_{L^2_f(\Hcal_s)} + \big\| (s/t) \del Z^I \vec{v} \big\|_{L^2_f(\Hcal_s)} + \big\| Z^I \vec{v} \big\|_{L^2_f(\Hcal_s)}
&\lesssim C_1 \eps,
\quad
&|I| \leq N-1.
\endaligned
$$
\end{lemma}

The bootstraps and Sobolev estimates \eqref{eq:Sobolev2} and \eqref{eq:Sobolev3} also immediately imply the following \textit{weak} decay estimates
\bel{eq:weak-decay-BA}
\aligned
|\del Z^I w|
&\lesssim C_1 \eps s^{-1+\delta},
\quad
&|I|&\leq N-2,
\\
|(s/t) \del Z^I \vec{v}|+ |Z^I \vec{v}|
&\lesssim C_1 \eps s^\delta t^{-1},
\quad
&|I|&\leq N-2,
\endaligned
\ee

\begin{lemma}[Strong decay estimates]\label{lem:decay-BA}
Let the estimates in the bootstrap assumptions \eqref{eq:BA-Easy} hold, then we have
$$
\aligned
|\del Z^I w|
&\lesssim C_1 \eps s^{-1},
\quad
&|I|&\leq N-3,
\\
|(s/t) \del Z^I \vec{v}| + |Z^I \vec{v}|
&\lesssim C_1 \eps t^{-1},
\quad
&|I| &\leq N-3,
\endaligned
$$
\end{lemma}

\begin{proof}
The results follow by combining the uniform energy bound in Lemma \ref{lem:L2-BA} with the Sobolev estimate of Lemma \ref{lem:sobolev} as well as the estimates for commutators from Lemma \ref{lem:est-comm}. 
\end{proof}

\subsection{Lower-order bootstraps and the auxiliary system}

To refine the lower-order bootstrap assumptions we make a decomposition on our variables. 
Recall our initial data is
$$
\big( w, \del_t w, \vec{v}, \del_t \vec{v} \big)(2) = (w_0, w_1, \vec{v}_0, \vec{v}_1).
$$
Following Katayama \cite{Katayama12a}, we introduce the decomposition
\bel{eq:w-decomp}
w = W_1 + P^\alpha\del_\alpha W_2.
\ee 
The variables $W_1, W_2, v^i$ are the solutions to the auxiliary equations
\bel{eq:aux-wave-map}
\aligned
-\Box W_1 &=  \mathcal{F}^0(w, \vec{v}),
\\
-\Box W_2 &= |\vec{v}|^2,
\\
-\Box v^i + v^i &= \widetilde{P}^\beta v^i \del_\beta w + \mathcal{F}^i(w, \vec{v}), \qquad i=1,\ldots, p,
\endaligned
\ee
with initial data
$$
\big( W_1, \del_t W_1, W_2, \del_t W_2, \vec{v}, \del_t \vec{v} \big)(2) = (w_0, w_1- P^0 |\vec{v}_0|^2, 0,0, \vec{v}_0, \vec{v}_1).
%(0,0,w_0, w_1-|\vec{v}_0|^2, \vec{v}_0, \vec{v}_1).
$$
%Solutions to \eqref{eq:aux-wave-map} satisfy the original initial value problem \eqref{eq:model-wave-map}. 
Using the bootstrap assumptions in \eqref{eq:BA-Easy} for the original variable $w$, we can derive estimates for the variables $W_1, W_2$. Once these are improved, we can then pass back to the original unknown $w$ by the relation \eqref{eq:w-decomp}.

%To refine the lower-order bootstrap assumptions we make a transformation on our variables following Katayama \cite{Katayama12a}. 
%Recall our initial data is
%$$
%\big( w, \del_t w, \vec{v}, \del_t \vec{v} \big)(2) = (w_0, w_1, \vec{v}_0, \vec{v}_1).
%$$
%Let $(W_1, W_2,v^i)$ solve the auxiliary  equations
%\bel{eq:aux-wave-map}
%\aligned
%-\Box W_1 &=  \mathcal{F}^0(W_1 + P^\alpha\del_\alpha W_2, \vec{v}),
%\\
%-\Box W_2 &= |\vec{v}|^2,
%\\
%-\Box v^i + v^i &= \widetilde{P}^\beta v^i \del_\beta (W_1 + P^\alpha\del_\alpha W_2) + \mathcal{F}^i(W_1 + P^\alpha\del_\alpha W_2, \vec{v}), \qquad i=1,\ldots, p,
%\endaligned
%\ee
%with initial data
%$$
%\big( W_1, \del_t W_1, W_2, \del_t W_2, \vec{v}, \del_t \vec{v} \big)(2) = (w_0, w_1- P^0 |\vec{v}_0|^2, 0,0, \vec{v}_0, \vec{v}_1).
%%(0,0,w_0, w_1-|\vec{v}_0|^2, \vec{v}_0, \vec{v}_1).
%$$
%Then the pair $(w,v^i)$, where $w$ is defined by
%\bel{eq:w-decomp}
%w \define W_1 + P^\alpha\del_\alpha W_2
%\ee 
%satisfy the original initial value problem \eqref{eq:model-wave-map}. 

\begin{proposition}\label{prop:lower-order-wave-sup}
Let the estimates in the bootstrap assumptions \eqref{eq:BA-Easy} hold. Then we have
$$ 
\aligned
\big\| (t-r) (s/t) \del Z^I w \big\|_{L^2_f(\Hcal_s)}
&\lesssim
\big(C_0\eps + (C_1 \eps)^2 \big) s^\delta, 
\quad &|I|\leq N-1,
\\
\big| \del Z^I w \big|
&\lesssim
\big(C_0\eps + (C_1 \eps)^2 \big) s^{-1+\delta} (t-r)^{-1}, 
\quad &|I|\leq N-3.
\endaligned
$$
\end{proposition}

\begin{proof}
We apply the first conformal-type energy estimate from Proposition \ref{lem:ConformalEnergy} to get, for $|I|\leq N$,
$$
\aligned
E_{con} (s, Z^I W_1)^{1/2}
&\lesssim 
E_{con} (s_0, Z^I W_1)^{1/2}
+
\int_{s_0}^s  \tau \big\| Z^I \mathcal{F}^0(w, \vec{v})  \big\|_{L^2_f(\Hcal_{\tau })} \, \di \tau.
%\\ &\lesssim 
%C_0\eps + (C_1 \eps)^3 s^\delta.
\endaligned
$$
We aim to show the bound 
\begin{equation}\label{eq_F0bd}
\big\| Z^I \mathcal{F}^0(w, \vec{v})  \big\|_{L^2_f(\Hcal_{\tau })} \lesssim (C_1 \eps)^3 \tau^{-2+\delta},
\qquad
|I|  \leq N.
\end{equation}
We only estimate the term $(A^{1})^{\alpha\beta}_i v^i \del_\alpha w \del_\beta w$ in $\mathcal{F}^0(w, \vec{v})$ since the others can be estimated in the same way. Note that a key feature of $\mathcal{F}^0$ is that each term contains at least one undifferentiated Klein-Gordon field. 
For $|I|  \leq N$ with $N \geq 3$ we find that
\begin{align*}
&\quad
\big\| Z^I \big( \vec{v} \del w \del w \big)\big\|_{L^2_f(\Hcal_{\tau })}
\\
&\lesssim
\sum_{\substack{|I_1|  \leq N\\ |I_2| + |I_3| \leq N-3}} \Big(\big\| Z^{I_1}  \vec{v} \big\|_{L^2_f(\Hcal_{\tau })} \big\| Z^{I_2}  \del w Z^{I_3} \del w \big)\big\|_{L^\infty(\Hcal_{\tau })}
\\
&\hskip4cm +
\big\| (\tau/t) Z^{I_1}  \del w \big\|_{L^2_f(\Hcal_{\tau })} \big\| (t/\tau) Z^{I_2}  \vec{v} Z^{I_3} \del w \big)\big\|_{L^\infty(\Hcal_{\tau })} \Big)
\\
&\quad
+
\sum_{\substack{|I_1|  \leq N-1\\ |I_2| + |I_3| \leq N-2}} \Big(\big\| Z^{I_1}  \vec{v} \big\|_{L^2_f(\Hcal_{\tau })} \big\| Z^{I_2}  \del w Z^{I_3}  \del w \big)\big\|_{L^\infty(\Hcal_{\tau })}
\\
&\hskip4cm +
\big\| (\tau/t) Z^{I_1}  \del w \big\|_{L^2_f(\Hcal_{\tau })} \big\| (t/\tau) Z^{I_2}  \vec{v} Z^{I_3}  \del w \big)\big\|_{L^\infty(\Hcal_{\tau })} \Big)
\\
&\lesssim
\big( C_1 \eps \big)^3 \tau^{-2+\delta},
\end{align*}
in which we used the commutator estimates as well as the observation that in each product at most one growing factor of $\tau^\delta$ appears from the $L^2_f$ part or the $L^\infty$ part.

The second estimate from Proposition \ref{lem:ConformalEnergy} then implies, for $|I|\leq N$,
\bel{eq:L2W1}
\aligned
\big\| (s/t) Z^I W_1 \big\|_{L^2_f(\Hcal_s)}
&\lesssim
\big\| (s_0/t) Z^I W_1 \big\|_{L^2_f(\Hcal_{s_0})}
+
\int_{s_0}^s \tau^{-1} E_{con} (\tau, Z^I W_1)^{1/2} \, \di \tau
\\
&\lesssim
C_0\eps + (C_1 \eps)^3 s^\delta.
\endaligned
\ee
The estimates in Proposition \ref{prop:L2} then allow us to obtain, for $|I|\leq N$,
$$
\aligned
\big\| (s/t) L_0 Z^I W_1 \big\|_{L^2_f(\Hcal_s)}
&\lesssim
\big\| (s/t) Z^I W_1 \big\|_{L^2_f(\Hcal_s)}
+ E_{con} (s, Z^I W_1)^{1/2}
\\ &\lesssim
C_0\eps + (C_1 \eps)^3 s^\delta. 
\endaligned
$$
The Sobolev inequality \eqref{eq:Sobolev3} and commutator estimates \eqref{eq:CommutatorsK} and \eqref{eq:est-cmt1} then provide us with pointwise decay
$$
\big| L_0 Z^I W_1 \big| + \big| L Z^I W_1 \big|
\lesssim \big( C_0\eps + (C_1 \eps)^3 \big) s^{-1+\delta}, \quad |I| \leq N- 2.
$$
Then, by \eqref{eq:202} or the identities,
\bel{deriv-identities}
\del_t
=
{t^2 \over s^2} \Big( t^{-1} L_0-\frac{x^a}{t^2} L_a \Big),
\qquad
\del_a
=
-{ x^a \over s^2} L_0 + {x^a x^b \over t^3} L_b + t^{-1} L_a,
\ee
we obtain, for $|I|\leq N-2$,
\bel{eq:supW1-10}
\aligned
\big| \del Z^I W_1 \big|
&\lesssim
(t-r)^{-1} \big( \big| L_0 Z^I W_1 \big| + \big| L Z^I W_1 \big| \big)
\\
&\lesssim
\big(C_0\eps + (C_1 \eps)^3 \big) s^{-1+\delta} (t-r)^{-1}.
\endaligned
\ee
By using \eqref{deriv-identities} and  $s^2 = (t-r)(t+r)$ on $\Hcal_s$, we also find
$$
\big\| (t-r) (s/t) \del Z^I W_1 \big\|_{L^2_f(\Hcal_s)}
\lesssim
C_0\eps + (C_1 \eps)^3 s^\delta,
\qquad
|I|\leq N.
$$

As for the $W_2$ component of the decomposition \eqref{eq:w-decomp}, we first obtain energy estimates when $|I|\leq N$
\bel{eq:energy-W2}
\aligned
E(s, Z^I W_2)^{1/2} 
&\leq E(s_0, Z^I W_2)^{1/2}  + \int_{s_0}^s \| Z^I |\vec{v}|^2\|_{L_f^2(\Hcal_\tau)} \di \tau
\\
&\lesssim C_0\eps +  \int_{s_0}^s (C_1\eps)^2 \tau^{-1+\delta} \di \tau
\\
& \lesssim C_0\eps + (C_1 \eps)^2 s^\delta.
\endaligned
\ee
Combining this estimate with the Sobolev estimate \eqref{eq:Sobolev3} yields a pointwise bound
\bel{eq:supW2-10}
|\del Z^I W_2| \lesssim (C_0\eps + (C_1 \eps)^2) s^{-1+\delta}, \qquad |I|\leq N-2.
\ee
Now by Lemmas \ref{lem:Wavedeldel} and \ref{lem:decay-BA}, decay estimates \eqref{eq:weak-decay-BA} and the commutator estimate \eqref{eq:est-cmt1}, we obtain, for $|I|\leq N-3$,
$$
\aligned
\big| \del \del Z^I W_2 \big|
&\lesssim
(t-r)^{-1} \Big( |\del L Z^I W_2 \big| + \big| \del Z^I W_2 \big| + t \big| Z^I (|\vec{v}|^2) \big| \Big)
\\
&\lesssim \big(C_0 \eps + (C_1 \eps)^2 \big) s^{-1+\delta} (t-r)^{-1} + (C_1\eps)^2 s^\delta t^{-1} (t-r)^{-1}
\\
&\lesssim 
\big(C_0\eps + (C_1 \eps)^2 \big) s^{-1+\delta} (t-r)^{-1} .
\endaligned
$$

By again using Lemma \ref{lem:Wavedeldel}, together with equations \eqref{eq:s-t-identities}, \eqref{eq:est-cmt1}, \eqref{eq:weak-decay-BA} and \eqref{eq:energy-W2} (and provided  $\lfloor\frac{N-1}{2}\rfloor\leq N-2$), we also obtain
$$
\big\| (t-r) (s/t) \del \del Z^I W_2 \big\|_{L^2_f(\Hcal_s)}
\lesssim
\big(C_0\eps + (C_1 \eps)^2 \big) s^\delta,
\qquad
|I|\leq N-1.
$$

We can now bring these estimates all together. Equations \eqref{eq:supW1-10}, \eqref{eq:supW2-10} and \eqref{eq:est-cmt1} imply, for $|I|\leq N-3$,
$$
\aligned
\big| \del Z^I w \big|
&\lesssim
\big| \del Z^I W_1 \big|
+
\big|P^\alpha \del Z^I \del_\alpha W_2 \big|
\\
&
\lesssim
\big(C_0\eps + (C_1 \eps)^2 \big) s^{-1+\delta} (t-r)^{-1}.
\endaligned
$$
Similarly, for $|I|\leq N-1$,
$$
\aligned
\big\| (t-r) (s/t) \del Z^I w \big\|_{L^2_f(\Hcal_s)}
&\lesssim
\big\|(t-r) (s/t) \del Z^I W_1 \big\|_{L^2_f(\Hcal_s)}
+
\big\| (t-r) (s/t) \del Z^I \del W_2 \big\|_{L^2_f(\Hcal_s)}
\\
&\lesssim
\big(C_0\eps + (C_1 \eps)^2 \big) s^{\delta}.
\endaligned
$$
\end{proof}

\begin{proposition}\label{prop:lower-order-wave-energy}
Let the estimates in the bootstrap assumptions \eqref{eq:BA-Easy} hold. Then we have
$$
E(s, Z^I w)^{1/2}
\lesssim C_0 \eps + (C_1 \eps)^{3/2},
\qquad
|I|\leq N-1.
$$
\end{proposition}

\begin{proof}

We act $Z^I$ with $|I|\leq N-1$ to the $w$ equation in \eqref{eq:model-wave-map}. The second energy estimate of Proposition \ref{prop:BasicEnergyEstimate} then implies
$$
\aligned
E (s, Z^I w)
&\lesssim
E (s_0, Z^I w)
\\
&\quad 
+
\int_{s_0}^s \left\| (\tau/t) \del_t Z^I w \cdot Z^I \big( P^\alpha \del_\alpha (|\vec{v}|^2) + \mathcal{F}^0(w, \vec{v})  \big) \right\|_{L^1_f(\Hcal_\tau)} \, d\tau.
\endaligned
$$

First note that \eqref{eq:s-t-identities}, Lemma \ref{lem:L2-BA}, \eqref{eq:weak-decay-BA} and $s^2 = (t-|x|)(t+|x|)$ imply
\begin{subequations}\label{eq:weighted-v-decay}
\begin{align}
\sup_{(t,x)\in\Hcal_s}|(t-|x|)^{-1} Z^I  \vec{v}(t,x)| 
&\lesssim (C_1 \eps) s^{-2+\delta}, \qquad |I|\leq N-2,  \label{eq:weighted-v-decay1}
\\
\sup_{(t,x)\in\Hcal_s}|(t-|x|)^{-1} \del Z^I \vec{v}(t,x)| 
&\lesssim (C_1 \eps) s^{-2+\delta}, \qquad |I|\leq N-3. \label{eq:weighted-v-decay2}
\end{align}
\end{subequations}
Using this, and provided that $\lfloor\tfrac{N-3}{2}\rfloor+1\leq N-2$, we observe that, for $|I|\leq N-1$,
$$
\aligned
& \left\| (t-r)^{-1} Z^I \big( P^\alpha \del_\alpha (|\vec{v}|^2) \big) \right\|_{L^2_f(\Hcal_\tau)}
\lesssim
\sum_{|I_1| \leq N-1} \big\| Z^{I_1} \vec{v} \big\|_{L^2_f(\Hcal_\tau)} \big\| (t-r)^{-1} \del \vec{v} \big\|_{L^\infty(\Hcal_\tau)} 
\\
&\quad
+ \sum_{\substack{|I_1| \leq N-1\\ |I_2| \leq N-2}} \big\| Z^{I_1} \del \vec{v} \big\|_{L^2_f(\Hcal_\tau)} \big\| (t-r)^{-1} Z^{I_2}\vec{v} \big\|_{L^\infty(\Hcal_\tau)}.
\endaligned
$$
Thus using the commutator estimates \eqref{eq:est-cmt1}, together with Lemma \ref{lem:L2-BA} and Proposition \ref{prop:lower-order-wave-sup}, we obtain   
\bel{eq:weighted-L1-unifw}
\aligned
\quad
\big\| &(\tau/t) \del_t Z^I w \cdot Z^I \big( P^\alpha \del_\alpha (|\vec{v}|^2) \big) \big\|_{L^1_f(\Hcal_\tau)}
\\
&\lesssim
\big\| (t-r) (\tau/t) \del_t Z^I w \big\|_{L^2_f(\Hcal_\tau)} \left\| (t-r)^{-1} Z^I \big( P^\alpha \del_\alpha (|\vec{v}|^2) \big) \right\|_{L^2_f(\Hcal_\tau)}
\\
&\lesssim
\big(C_0 \eps + (C_1 \eps)^2 \big) \tau^{\delta} \cdot (C_1 \eps \tau^\delta) \cdot (C_1\eps  \tau^{-2+\delta}).
\endaligned
\ee

We now turn to the first term in $\mathcal{F}^0$. Provided $\lfloor \frac{N-1}{2}\rfloor \leq N-2$ we find, for $|I|\leq N-1$,
$$
\aligned
&\left\| (t-r)^{-1} Z^I (v^i \del w \del w) \right\|_{L^2_f(\Hcal_\tau)}
\lesssim
 \sum_{|I_1| \leq N-1} \big\| Z^{I_1} \vec{v} \big\|_{L^2_f(\Hcal_\tau)}  \sum_{|I_2| \leq N-2} \big\|Z^{I_2} \del w \big\|_{L^\infty(\Hcal_\tau)} ^2
\\
&\qquad + \sum_{|I_1| \leq N-1} 
\big\|(\tau/t) Z^{I_1} \del w \big\|_{L^2_f(\Hcal_\tau)} 
\sum_{|I_2| \leq N-2}\big\|(t/\tau) Z^{I_2}(v^i \del w) \big\|_{L^\infty(\Hcal_\tau)}.
\endaligned
$$
So using the commutator estimates \eqref{eq:est-cmt1}, together with Lemma \ref{lem:L2-BA}, the weak-decay estimates \eqref{eq:weak-decay-BA} and Proposition \ref{prop:lower-order-wave-sup}, we find
$$
\aligned
&\left\| (\tau/t) \del_t Z^I w \cdot Z^I \big( v^i \del w \del w \big) \right\|_{L^1_f(\Hcal_\tau)}
\\
&\quad \lesssim 
\big\| (t-r) (\tau/t) \del_t Z^I w \big\|_{L^2_f(\Hcal_\tau)} 
\| (t-r)^{-1} Z^I (v^i \del w \del w) \|_{L^2_f(\Hcal_\tau)}
\\
&\quad \lesssim
\big( C_0\eps + (C_1 \eps)^2  \big) \tau^\delta \cdot (C_1 \eps) \cdot (C_1 \eps \tau^{-2+2\delta}).
\endaligned
$$

The other terms in $\mathcal{F}^0$ are easier to estimate, and so we obtain, for $|I| \leq N-1$,
$$
\left\| (\tau/t) \del_t Z^I w \cdot Z^I \big( P^\alpha \del_\alpha (|\vec{v}|^2) + \mathcal{F}^0(w, \vec{v})  \big) \right\|_{L^1_f(\Hcal_\tau)}
\lesssim
(C_1 \eps)^3 \tau^{-2+3\delta}.
$$
Since $\delta \ll1 $, by the energy inequality above we have 
$$
E (s, Z^I w)
\lesssim
(C_0 \eps)^2 + (C_1 \eps)^3,
$$
and this finishes the proof.
\end{proof}

\begin{remark}\label{rem:nonlinear-structure}
In the above proposition, it is very important that we use the $L^1$ version of the energy inequality in order to distribute certain $(t-r)$ weights across the terms, see \eqref{eq:weighted-L1-unifw}. We are also careful when distributing commutators across the quadratic nonlinearity  $v \del v$. In particular, when at least one commutator hits the $\del v$ component, we absorb terms like $\| Z^I \del \vec{v}\|_{L^2_f(\Hcal_s)}$, for $|I|\leq N-1$, by the `mass' term contained in the top-order energy $E(s, Z^I v)^{1/2}$ with $|I|\leq N$. Although this produces some $s^\delta$ growth, by taking the $\del v$ term in $L^2$, we can then use the strong pointwise decay estimate \eqref{eq:weighted-v-decay1} on $v$. This also means we require less regularity than if we had used  \eqref{eq:weighted-v-decay2}. Note that this approach would fail if the nonlinearity took the form $\del v \del v$ since we there would be no benefit to distributing $(t-r)$ weights across the components, and higher regularity assumptions on the initial data are need to make it work.
\end{remark}

\begin{proposition}\label{prop:lower-order-KG-energy}
Let the estimates in the bootstrap assumptions \eqref{eq:BA-Easy} hold. Then we have
$$
E_1 (s, Z^I \vec{v})^{1/2}
\lesssim  C_0\eps + (C_1 \eps)^{3/2},
\qquad
|I|\leq N-1.
$$
\end{proposition}
\begin{proof}
We consider the second energy estimate \eqref{eq:w-EE-L1} from Proposition \ref{prop:BasicEnergyEstimate}. This gives
$$
\aligned
E_1(s, Z^I v^i) 
& \lesssim E_1(s_0, Z^I v^i) + 
\int_{s_0}^s \| (\tau/t)\del_t Z^I v^i \cdot Z^I (-\Box v^i + v^i) \|_{L_f^1(\Hcal_\tau)}\di \tau.
\endaligned
$$
We first estimate the quadratic term in the nonlinearity. 
Using \eqref{eq:BA-Easy3} and $s^2 = (t-|x|)(t+|x|)$ we see
\bel{eq:weighted-w-decay}
\sup_{(t,x)\in \Hcal_s} |(s/t)\del Z^I w(t,x)|
\lesssim C_1 \eps s^{-2+\delta} ,
\qquad
|I| \leq N-3.
\ee
Using \eqref{eq:weighted-w-decay}, and provided that $\lfloor \tfrac{N-3}{2}\rfloor+1\leq N-2$, we find, for $|I|\leq N-1$,
$$ 
\aligned
\big\| &(\tau/t) Z^I (v \del w )\big\|_{L_f^2(\Hcal_\tau)}
\lesssim
\sum_{|I_1| \leq N-1}
\| Z^{I_1} \vec{v}\|_{L_f^2(\Hcal_\tau)} \| (\tau/t)\del w \|_{L^\infty(\Hcal_\tau)}
\\
&\qquad +\sum_{\substack{|I_1| \leq N-2,\\ |I_2|\leq N-1}}
\| (t-r)^{-1}Z^{I_1} \vec{v}\|_{L^\infty(\Hcal_\tau)} \| (t-r) (\tau/t)Z^{I_2} \del w \|_{L^2_f(\Hcal_\tau)} .
\endaligned
$$
Thus using the commutator estimates  \eqref{eq:est-cmt1}, together with the decay estimates \eqref{eq:weighted-v-decay1} and \eqref{eq:weighted-w-decay}, Lemma \ref{lem:L2-BA} and Proposition \ref{prop:lower-order-wave-sup}, we obtain
$$ 
\aligned
&\quad
\big\| (\tau/t)\del_t Z^I v^i \cdot Z^I (v \del w )\big\|_{L_f^1(\Hcal_\tau)}
\\
&\lesssim
\sum_{|I|\leq N-1} \| \del Z^I \vec{v}\|_{L^2(\Hcal_\tau)} 
\sum_{|I|\leq N-1}  \big\| (\tau/t) Z^I (v \del w )\big\|_{L_f^2(\Hcal_\tau)}
\\
&\lesssim
(C_1 \eps \tau^\delta) \cdot (C_1 \eps \tau^\delta) \cdot (C_1 \eps \tau^{-2+\delta}).
\endaligned
$$
Similar to the estimate \eqref{eq_F0bd} proven already, for $|I|\leq N-1$ we have
$$ 
\| Z^I \mathcal{F}^i(w, \vec{v})\|_{L^2(\Hcal_\tau)} \lesssim (C_1 \eps)^3 \tau^{-2}.
$$
Putting these together we find
$$ 
\aligned
E(s, Z^I v^i) 
& \lesssim E(s_0, Z^I v^i) + 
\int_{s_0}^s (C_1 \eps)^3 \tau^{-2+3\delta}\di \tau
\lesssim (C_0\eps)^2 + (C_1\eps)^3,
\endaligned
$$
and thus
$$
E(s, Z^I v^i)^{1/2} \lesssim C_0\eps + (C_1 \eps)^{3/2}, \qquad |I|\leq N-1.
$$
\end{proof}

\subsection{Top-order bootstraps and proof of Theorem \ref{thm1}}

\begin{proposition}\label{prop:high-order-wave-KG-energy}
Let the estimates in the bootstrap assumptions \eqref{eq:BA-Easy} hold. Then we have
$$
E(s, Z^I w)^{1/2} + E_1 (s, Z^I v^i)^{1/2}
\lesssim C_0\eps + (C_1 \eps)^2 s^\delta,
\quad
|I|\leq N.
$$
\end{proposition}

\begin{proof}
As in the proof of the previous two propositions, we make use of the lower-order uniform energy bounds in Lemma \ref{lem:L2-BA} and the weak decay estimates \eqref{eq:weak-decay-BA} to reduce the required regularity as much as possible. 
We find for $|I|\leq N$
$$
\aligned
\| Z^I \mathcal{F}^0(w, \vec{v})\|_{L_f^2(\Hcal_s)}
+ \| Z^I \mathcal{F}^i(w, \vec{v})\|_{L_f^2(\Hcal_s)} &\lesssim (C_1 \eps)^3 s^{-2+\delta} ,
\\
\| Z^I P^\alpha \del_\alpha(|\vec{v}|^2) \|_{L_f^2(\Hcal_s)}
&\lesssim (C_1 \eps)^2 s^{-1+\delta},
\\
\| Z^I (\widetilde{P}^\alpha v^i \del_\alpha w ) \|_{L_f^2(\Hcal_s)}
&\lesssim (C_1 \eps)^2 s^{-1+\delta}.
\endaligned
$$
The conclusion then follows by the first energy estimate of Proposition \ref{prop:BasicEnergyEstimate}. 
\end{proof}

\begin{proof}[Proof of Theorem \ref{thm1}]
We can now bring together all the components of the bootstrap argument to conclude the main theorem. Firstly, as shown in \cite[\textsection11]{PLF-YM-book}, initial data posed on the hypersurface $\{t_0=2\}$ and localised in the unit ball $\{x\in\RR^2:|x|\leq 1\}$ can be developed as a solution of the PDE to the initial hyperboloid $\Hcal_{s_0=2}$ with the smallness conserved. This justifies the bound \eqref{eq:BApre}. Next, by classical local existence results for quasilinear hyperbolic PDEs, the bounds \eqref{eq:BA-Easy} hold whenever the solution exists. Clearly $s_1>s_0$, and if $s_1<+\infty$ then one of the inequalities in \eqref{eq:BA-Easy} must be an equality. We see then that Propositions \ref{prop:high-order-wave-KG-energy}, \ref{prop:lower-order-wave-sup}, \ref{prop:lower-order-wave-energy} and \ref{prop:lower-order-KG-energy} imply that by choosing $C_1$ sufficiently large and $\eps$ sufficiently small, the bounds \eqref{eq:BA-Easy} are in fact refined. This then implies that $s_1=+\infty$ and so the local solution extends to a global one. The decay estimates \eqref{eq:thm-decay}  are shown via Lemma \ref{lem:decay-BA}. The decay estimate \eqref{eq:thm-decay2} is easily obtained by applying the Sobolev inequality \eqref{eq:Sobolev3} to the estimate \eqref{eq:L2W1} and combining this with \eqref{eq:supW2-10}. 
\end{proof}

\end{document}